\documentclass[11pt]{amsart}
\input epsf
\usepackage{latexsym}
\usepackage{amssymb,amsmath,amsthm,amscd, amssymb, 
graphicx, enumerate, verbatim, bm, 
accents}
\usepackage[all]{xy}

\numberwithin{equation}{section}

\newtheorem{lem}[equation]{Lemma}

\newtheorem{thm}[equation]{Theorem}

\newtheorem{Example}[equation]{Example}

\newtheorem{remark}[equation]{Remark}
\newenvironment{rmk}{\begin{remark}\rm}{\end{remark}}
\def\co{\colon\thinspace}

\newcommand{\Int}{\mbox{Int}}

\newcommand{\vol}{\mbox{vol}}

\newcommand{\Iso}{\mbox{Iso}}

\newcommand{\e}{\varepsilon}

\newcommand{\II}{\mathbb I}
\def\bu{\bullet}
\def\a{\alpha}
\def\G{\Gamma}
\def\B{\mathcal B}
\def\de{\delta}
\def\g{\gamma}
\def\o{\omega}

\def\Si{\Sigma}

\def\M{\mathcal M}
\def\K{\mathcal K}

\def\sf{\mathfrak s}
\def\L{\Lambda}

\def\d{\partial}
\def\k{\kappa}

\def\r{\rho}
\def\o{\omega}

\def\s{\sigma}

\def\Cka{\scalebox{.9}{$C^{k,\a}\!$}}
\def\Czo{\scalebox{.9}{$C^{\hspace{.5pt}0,1}\!$}}

\def\Ck{\scalebox{.9}{$C^{k}\!$}}

\def\Ckz{\scalebox{.9}{$C^{k,0}\!$}}

\def\Cinfty{\scalebox{.9}{$C^{\infty}\!$}}

\def\S1{\bf S^1}

\newcommand{\R}{{\mathbb R}}

\makeatletter
\def\equalsfill{$\m@th\mathord=\mkern-7mu
\cleaders\hbox{$\!\mathord=\!$}\hfill
\mkern-7mu\mathord=$}
\makeatother

\begin{document}

\abovedisplayskip=6pt plus3pt minus3pt
\belowdisplayskip=6pt plus3pt minus3pt

\title[Hyperspaces of smooth convex bodies up to  congruence]
{\bf Hyperspaces of smooth convex bodies \\ up to congruence}

\rm

\author{Igor Belegradek}

\address{Igor Belegradek\\School of Mathematics\\ Georgia Institute of
Technology\\ Atlanta, GA 30332-0160}\email{ib@math.gatech.edu}


\date{}
\begin{abstract} 
We determine the homeomorphism 
type of the hyperspace of positively curved $C^\infty$ convex bodies 
in $\mathbb R^n$, and derive various properties of its quotient by the group of Euclidean isometries. 
We make a systematic study of 
hyperspaces of convex bodies that are at least $C^1$. 
We show how to destroy the symmetry of a family of convex bodies,
and prove that this cannot be done modulo congruence.
\end{abstract}
\maketitle

\section{Introduction}
\label{sec: intro}

A {\em hyperspace of\,} $\R^n$ is a set of compact subsets of $\R^n$
equipped with the Hausdorff metric, 
and two subsets are {\em congruent\,} if they lie
in the same orbit of $\Iso(\R^n)$, the group of Euclidean isometries. 
To avoid trivialities {\em we assume that $n\ge 2$.} 

A {\em convex body in $\R^n$\,} is a compact convex set with nonempty interior.
A function is $\Cka$ if its
$k$th partial derivatives are
$\a$-H\"older for $\a\in (0,1]$ and continuous for $\a=0$ where
as usual $\Ck$ means $\Ckz$, see~\cite{GilTru}
and~\cite[Section 2]{BanBel} for background.
A {\em $\Cka$ convex body\,} is a convex body whose boundary is 
a $\Cka$ submanifold of $\R^n$.
Any convex body is $\Czo$ because convex functions are locally Lipschitz.

There is an established framework for studying
topological properties of hyperspaces, and e.g., the homeomorphism types
of the following hyperspaces of $\R^n$ are known: 
convex compacta~\cite{NQS}, convex bodies~\cite{AntPer}, 
convex polyhedra~\cite[Exercise 4.3.7]{BRZ-book}, 
strictly convex bodies~\cite{Baz-strict-conv},
convex compacta of constant width~\cite{Baz-width, BazZar-width, AJJ-width}.

One goal of this paper is to add to this list 
the hyperspace of $\Cinfty$ convex bodies of positive Gaussian curvature.

Another goal is to study certain hyperspaces that are not closed 
under Minkowski sum, e.g., the results of this paper
are used in~\cite{Bel-ghs2}
to determine the homeomorphism type of a hyperspace of $\R^3$
whose $\Iso(\R^3)$-quotient is homeomorphic to the Gromov-Hausdorff space 
of $\Cinfty$ nonnegatively curved $2$-spheres. 

Let $\K_s$ be the hyperspace of convex compacta in $\R^n$ 
with Steiner point at the origin, and let $\B_p$ be the hyperspace
of $\Cinfty$ convex bodies in $\K_s$ with boundary of 
positive Gaussian curvature. Placing the Steiner point at the origin
is mainly a matter of convenience. In particular, the space of all 
convex compacta $\K$ in $\R^n$ is homeomorphic to
$\K_s\times\R^n$, and the orbit spaces $\K/\mathrm{Iso}(\R^n)$, $\K_s/O(n)$
are homeomorphic, see (\ref{form: st pt times Rn})
and Lemma~\ref{lem: st point}.

Consider the {\em Hilbert cube\,} $Q=[-1, 1]^\o$ and its {\em radial interior\,} 
$$\Si=\{(t_i)_{i\in \o}\ \text{in}\ Q\,:\, \displaystyle{\sup_{i\in\o}|t_i| < 1} \}.$$
The superscript $\o$ refers to the product of countably many copies of a space.
We have a canonical inclusion $\Si^\o\subset Q^{\,\o}$.
Clearly $Q^{\,\o}$ and $Q$ are homeomorphic. Also $\Si$ is $\s$-compact
while $\Si^\o$ is not.
Here is the main result of this paper.

\begin{thm} 
\label{thm: intro hyperspace}
Given a point $*$ in $Q^{\,\o}\setminus\Sigma^\o$
there exists a homeomorphism $\K_s\to Q^{\,\o}\setminus \{*\}$
that takes $\B_p$ onto $\Si^\o$.
\end{thm}

The proof of Theorem~\ref{thm: intro hyperspace} verifies the assumptions 
of recognition theorems for spaces modelled on $Q$ and $\Si^\o$ 
as described e.g., in~\cite{BRZ-book}.  
This involves various convex geometry techniques, as well as
a method developed in~\cite{BanBel}.   

Yet another objective of this paper is to study the topology of $\K_s$
and $\B_p$ modulo congruence, see~\cite{TorWes, AgeBogRep, Ant-fm2000, AntPer, Age2016} 
for related results. 

Set $\bm{\K}_s=\K_s/O(n)$ and $\bm{\B}_p=\B_p/O(n)$ with the quotient topology. 
Denote the principal $O(n)$-orbits in $\K_s$, $\B_p$ by
$\accentset{\circ}{\K}_s$, $\accentset{\circ}{\B}_p$, respectively, and
let $\accentset{\circ}{\bm{\K}}_s$, $\accentset{\circ}{\bm{\B}}_p$
be their $O(n)$-orbit spaces. 
By the Slice Theorem  
$\accentset{\circ}{\bm{\K}}_s$ is open in $\bm{\K}_s$ while
the orbit maps $\accentset{\circ}{\K}_s\to \accentset{\circ}{\bm{\K}}_s$ and
$\accentset{\circ}{\B}_p\to \accentset{\circ}{\bm{\B}}_p$ are principal $O(n)$-bundles.
Some other properties of the spaces are summarized below.

\begin{thm}
\label{thm: intro quotient}
\begin{enumerate}
\item[\textup{(1)}]
$\bm{\K}_s$ is a locally compact Polish absolute retract. \vspace{2pt}
\item[\textup{(2)}] 
$\bm{\B}_p$ is an absolute retract that is neither Polish nor locally compact. 
\vspace{2pt}
\item[\textup{(3)}]
Any $\s$-compact subset of $\bm{\B}_p$ has empty interior.
\vspace{2pt}
\item[\textup{(4)}]
$\bm{\B}_p$ is homotopy dense in $\bm{\K}_s$, i.e., any continuous map
$Q\to \bm{\K}_s$ can be uniformly approximated by a continuous map with image in $\bm{\B}_p$. 
\vspace{2pt}
\item[\textup{(5)}]
$\accentset{\circ}{\B}_p$ and $\accentset{\circ}{\K}_s$ are contractible, while
$\accentset{\circ}{\bm{\B}}_p$ and $\accentset{\circ}{\bm{\K}}_s$ 
are homotopy equivalent to $BO(n)$, the Grassmanian of $n$-planes in $\R^\o$.
\vspace{2pt}
\item[\textup{(6)}]
The pairs $(\accentset{\circ}{\bm{\K}}_s, \accentset{\circ}{\bm{\B}}_p)$
and $(Q^{\,\o}, \Si^\o)$ are locally homeomorphic, i.e., 
each point of $\accentset{\circ}{\bm{\B}}_p$ has a neighborhood 
$U\subset\accentset{\circ}{\bm{\K}}_s$ such that some open embedding
$h\co U\to Q^{\,\o}$ takes 
$U\cap \accentset{\circ}{\bm{\B}}_p$ onto $h(U)\cap\Si^\o$.
\vspace{2pt}
\item[\textup{(7)}]
There is a locally finite simplicial complex $L$ and a homeomorphism 
$h\co \accentset{\circ}{\bm{\K}}_s\to L\times Q^{\,\o}$
that maps $\accentset{\circ}{\bm{\B}}_p$ onto  $L\times \Si^\o$. 
\end{enumerate}
\end{thm}

That $\bm{\K}_s$ and $\bm{\B}_p$ are absolute retracts is essentially due to 
Antonyan~\cite{Ant-ANE-G-conv, Ant-orbit-sp-ANR}. 
Homotopy density of $\bm{\B}_p$ in $\bm{\K}_s$ is immediate from Schneider's regularization of convex bodies, 
see Lemma~\ref{lem: regularization}. 

Contractibility of $\accentset{\circ}{\B}_p$, $\accentset{\circ}{\K}_s$
is established geometrically in Lemma~\ref{lem: hom dense trivial isometry}, 
which proves that these spaces are homotopy dense in $\K_s$.
Therefore, they are classifying spaces for principal $O(n)$-bundles,
and (5) of Theorem~\ref{thm: intro quotient} follows. 

The claim (6) of Theorem~\ref{thm: intro quotient}
exploits the $O(n)$-bundle structure and depends on
Theorem~\ref{thm: intro hyperspace}.

The claim (7) follows in a standard way from (5)-(6) and the observation
that $\accentset{\circ}{\bm{\K}}_s$ is homeomorphic to $\accentset{\circ}{\bm{\K}}_s\times [0,1)$,
see Lemma~\ref{lem: prod with [0,1)}. One can take $L$ to
be the product of $[0,1)$ with any locally finite simplicial complex
that is homotopy equivalent to $BO(n)$. 

Since $\bm{\K}_s$ is contractible while $\accentset{\circ}{\bm{\K}}_s$ is not,
the latter is not homotopy dense in the former, i.e., there is no
continuous ``destroy the symmetry'' map $\bm{\K}_s\to \accentset{\circ}{\bm{\K}}_s$
that would instantly push every singular $k$-disk in $\bm{\K}_s$ 
into $\accentset{\circ}{\bm{\K}}_s$. More precisely, 
such a map exists for $k\le 1$ but not for $k=2$,
see Section~\ref{sec: destroy} where we also prove a local version of this assertion.

The following questions highlight how much we do not yet know.
\begin{enumerate}
\vspace{-2pt}
\item[(a)]
Is the hyperspace $\B^\infty$ of $C^\infty$ convex bodies homeomorphic to $\Si^\o$?
Unlike $\B_p$ the hyperspace $\B^\infty$ is not convex~\cite{Bow-analy}, and convexity was
essential in our proof of strong $\M_2$-universality of $\B_p$. 
\vspace{3pt}
\item[(b)]
Are the  orbit spaces $\bm{\B}_p$ and $\bm{\K}_s$ topologically homogeneous, i.e.,
do their homeomorphism groups act transitively? (This seems unlikely).\vspace{3pt}
\item[(c)]
Is the congruence class of the unit sphere a $Z$-set in $\bm{\K}_s$?
Does it have contractible complement?\vspace{3pt}
\item[(d)]
Does every point $\bm{\B}_p$, $\bm{\K}_s$ 
have a basic of contractible neighborhoods? Like any AR,
these orbit spaces are locally contractible, i.e.,  
any neighborhood of a point contains a smaller neighborhood that
contracts inside the original one.\vspace{3pt}
\item[(e)]
As we shall see in Lemma~\ref{lem: Ks(Bn)} the hyperspace
$\K_s(\mathbb B^n)$ that consists of sets in $\K_s$
contained in $\mathbb B^n$ is homeomorphic to $Q$.
Is there an $O(n)$-equivariant homeomorphism of $\K_s(\mathbb B^n)$ 
onto a countable product of Euclidean units disks (of various dimensions)
where the $O(n)$-action on the product is diagonal and irreducible on each factor?
Such products are considered in~\cite[Section 1]{Ant-orbit-sp-1988},~\cite[p.553]{West-open-pr},~\cite[p.161]{Age2016}.
\vspace{3pt}
\item[(f)]
For $\k>0$ consider the hyperspaces $\B_{>\k}$, $\B_{\ge k}$ in $\K_s$ consisting
of of convex bodies whose boundary has Gaussian 
curvature $>k$ or $\ge k$, respectively. Are they ANR, or more generally 
$H$-ANR for every closed subgroup $H\le O(n)$?
All I can say is that $\B_{>\k}/H$ and 
$\B_{\ge\k}/H$ are weakly contractible
as each singular sphere in $\B_{\ge\k}/H$ contracts in $\B_p/H$ and
every singular disk in $\B_p/H$ can be rescaled into $\B_{>\k}$.
\end{enumerate}

The structure of the paper is as follows. 
Section~\ref{sec: intro} describes the main results and also lists some open questions.
In Section~\ref{sec: notations} we collect a number of notations and conventions.
Some facts of infinite dimensional topology and convex geometry are reviewed
in Sections~\ref{sec: inf dim top} and~\ref{sec: conv geom}, respectively.
In Section~\ref{sec: hyperspace of dim >1} we classify the $Q$-manifolds encountered in the paper.
Section~\ref{sec: hyp homeo to Si-to-o} is the heart of the paper where the
key claim of Theorem~\ref{thm: intro hyperspace} is proven: $\B_p$
is homeomorphic to $\Si^\o$.
The difficulty here is to match the tools of convex geometry 
with what is required by the infinite
dimensional topology.
The proof is finished in Section~\ref{sec: pairs} via a standard argument.
In Sections~\ref{sec: quotients}--\ref{sec: destroy} we
prove Theorem~\ref{thm: intro quotient} and show that one can continuously
destroy the symmetry of convex bodies, but one cannot do this modulo congruence. 

\section{Notations and conventions}
\label{sec: notations}

Throughout the paper $\o$ is the set of the nonnegative integers, $\II=[0,1]$, 
$\mathbb S^{n-1}=\d\mathbb B^n$ where $\mathbb B^{n}$ 
is the closed unit ball about the origin $o\in\R^n$. 
We use the following notations for hyperspaces of $\R^n$: 
\begin{gather*}
\K = \{\text{convex compacta in $\R^n$}\}\\
\K^{k\le l} =\{\text{sets in}\ \K 
\text{\ of dimension at least $k$ and at most $l$}\} \ \ \text{and}\  \
\K^k=\K^{k\le k}
\\ 
\K_s =\{\text{sets in}\ \K \ \text{with Steiner point at $o$}\}
\ \ \text{and}\  \
\K_s^{k\le l} =\K^{k\le l} \cap \K_s 
\\
\B^{k,\a} =\{
\text{$C^{k,\a}$ convex bodies in $\K_s$}
\} \ \ \text{and}\ \
\B^{k} =\B^{k,0}  
\\
\B_p  =\{
\text{convex bodies in  ${\mathcal  B}^\infty$ with boundary of
positive Gaussian curvature}\}
\end{gather*}
To stress that these are hyperspaces of $\R^n$
we may write $\K(\R^n)$ for $\K$, etc. On one occasion
we discuss $\K_s(\mathbb B^n)$, the hyperspace of consisting of
sets of $\K_s$ lying in the unit ball about the origin.
Note that $\B_p$ equals 
the class $C^\infty_+$ discussed in~\cite[Section 3.4]{Sch-book}.
Each of the above hyperspaces is $O(n)$-invariant, and we
denote the principal $O(n)$-orbit, i.e., the set of points with the trivial
isotropy in $O(n)$, by placing $\circ$ over the hyperspace symbol, e.g., 
$\accentset{\!\circ}{\B}_d$ is the principal orbit for the $O(n)$-action on $\B_p$.

We denote the $O(n)$-orbit space of a hyperspace by the same symbol made bold, e.g.,
$\bm{\B}^{k,\a}$, $\bm{\K}_s$, $\accentset{\,\circ}{\bm{\B}}_p$
are the $O(n)$-orbit spaces of 
${\B}^{k,\a}$, ${\K}_s$, $\accentset{\!\!\circ}{\B}_p$,
respectively.

\section{Brief dictionary of infinite dimensional topology}
\label{sec: inf dim top}

All definitions and notions of infinite 
dimensional topology that are used in this paper 
can be found in~\cite{BRZ-book} 
and are also reviewed below.

Unless stated otherwise any {\em space\,} is 
assumed metrizable and separable, 
and any {\it map\,} is assumed continuous.
A space is {\em Polish\,} if it admits a complete metric.
A {\em subspace\,} is a subset with subspace topology.
If $\Omega$ and $M$ are spaces, then $M$ is an {\em $\Omega$-manifold\,} 
if each point of $M$ has a neighborhood homeomorphic to an open subset of $\Omega$.

A closed subset $B$ of a space $X$ is a {\it $Z$-set\,} if every map $Q\to X$
can be uniformly approximated by a map whose range misses $B$. 
A {\em $\s Z$-set\,} is a countable union of $Z$-sets.
An embedding is a {\it $Z$-embedding\,} if its image is a $Z$-set.

A subspace $A\subset X$ is {\it homotopy dense\,} if there is a homotopy
$h\co X\times \II\to X$ with $h_0=\mathrm{id}$ and 
$h(X\times (0,1])\subset A$.
If $X$ is an ANR, then $A\subset X$ is homotopy dense 
if and only if each map $\II^k\to X$ with $k\in\o$ and $\d \II^k\subset B$ 
can be uniformly approximated
rel boundary by maps $\II^k\to B$~\cite[Theorem 1.2.2]{BRZ-book}. 

If $X$ is an ANR and $B\subset X$ is a closed subset, 
then $B$ is a $Z$-set if and only if 
$X\setminus B$ is homotopy dense~\cite[Theorem 1.4.4]{BRZ-book},~\cite[Proposition V.2.1]{BP-book}.

Given an open cover $\mathcal U$ of a space $X$ two maps 
$f,g\co Y\to X$ are {\it $\mathcal U$-close\,} 
if for every $y\in Y$ there is $U\in\mathcal U$ with $f(y), g(y)\in U$.

A space $X$ has the {\em Strong Discrete Approximation Property\,}
or simply {\em SDAP\,}
if for every open cover $\mathcal U$ of $X$ each map $Q\times \o\to X$
is $\mathcal U$-close to a map $g\co Q\times \o\to X$ such that every point
of $X$ has a neighborhood that intersects at most one set of the family 
$\{g(Q\times\{n\})\}_{n\in\o}$. 

A space $X$ has the {\em Locally Compact Approximation Property\,}
or simply {\em LCAP\,} if for every open cover $\mathcal U$ of $X$
there exists a map $f\co X\to X$ that is $\mathcal U$-close to the identity of $X$
and such that $f(X)$ has locally compact closure.

Let $\M_0$ be the class of compact spaces, and $\M_2$ be
the class of spaces homeomorphic to 
$F_{\sigma\delta}$-sets in compacta, see~\cite[Exercise 3 in 2.4]{BRZ-book}. 
Note that $X\in\M_2$ if and only if the image of any embedding of $X$
into a Polish space is $F_{\s\de}$, see~\cite[Theorem VIII.1.1]{BP-book}. 

Let $\mathcal C$ be a class of spaces, such as $\M_0$ or $\M_2$.
A space $X$ is {\em $\mathcal C$-universal\,} if 
each space in $\mathcal C$ is homeomorphic to a closed subset of $X$.

A space $X$ is {\it strongly $\mathcal C$-universal\,} if 
for every open cover $\mathcal U$ of $X$, 
each $C\in\mathcal C$, every closed subset $B\subset C$, and each
map $f\co C\to X$ that restricts to a 
\mbox{$Z$-embedding} on $B$ there is a $Z$-embedding
$\bar f\co C\to X$ with $\bar f\vert_B=f\vert_B$  such that $f$, 
$\bar f$ are $\mathcal U$-close.

A space $X$ is {\it $\mathcal C$-absorbing\,}
if $X$ is a strongly $\mathcal C$-universal ANR with SDAP
that is the union of countably many $Z$-sets, and 
also the union of a countably many closed subsets 
homeomorphic to spaces in $\mathcal C$.

For example, $\Si$ is $\M_0$-absorbing and $\Si^\o$ is $\M_2$-absorbing,
see~\cite[Exercises 3 in 1.6 and 2.4]{BRZ-book}. Let us list
some properties of $\M_2$-absorbing spaces:
\begin{itemize}
\item (Triangulated $\Si^\o$-manifold)\quad
A space $X$ is $\M_2$-absorbing if and only if $X$ is an {$\Si^\o$-manifold}
if and only if $X$ is homeomorphic to $\Si^\o\times L$
where $L$ is a locally finite simplicial complex~\cite[Corollary 5.6]{BesMog}.
\vspace{2pt}
\item (Uniqueness) Any two homotopy equivalent $\M_2$-absorbing spaces
are homeomorphic, see~\cite[Theorem 3.1]{BesMog}.
\vspace{2pt}
\item ($Z$-set unknotting) If $A$, $B$ are $Z$-sets in a $\M_2$-absorbing
space $X$, then any homeomorphism $A\to B$ that is homotopic to the inclusion
of $A$ into $X$ extends to a homeomorphism of $X$~\cite[Theorem 3.2]{BesMog}.
\end{itemize}

\section{Convex geometry background}
\label{sec: conv geom}

Our main reference for convex geometry is~\cite{Sch-book}.
We give $\R^n$ the Euclidean norm  
$\|v\|=\sqrt{\langle v,v\rangle}$ where
$\langle u,v\rangle=\sum_{i=1}^n u_i v_i$.

The {\em support function\,} $h_D\co \R^n\to\R$ of 
a compact convex non-empty set 
$D$ is defined by $h_D(v)=\sup\{\langle x,v\rangle\,:\, x\in D\}$.
Thus $h_{D+w}(v)=h_D(v)+\langle v,w\rangle$ for any $w\in \R^n$, hence
$D$ has nonempty interior if and only if there is $w\in \R^n$ 
such that $h_D(v)+\langle v,w\rangle>0$ for any $v\neq 0$.
The function $h_D$ is sublinear, i.e., $h_D(tv)=th_D(v)$ for $t>0$
and $h_D(v+w)\le h_D(v)+h_D(w)$. 
Conversely, any sublinear real-valued function on $\R^n$ is the support function of a unique compact
convex set in $\R^n$, see~\cite[Theorem 1.7.1]{Sch-book}.

If $o\in D$ and $u\in \mathbb S^{n-1}$, then $h_D(u)$ is the distance to the origin 
from the support hyperplane to $D$ with outward normal vector $u$. If $o\in\Int\, D$, then 
$h_D(v)>0$ for any $v\neq 0$. In summary, 
{\em support functions of convex bodies whose interior contains $o$
are precisely the sublinear positive functions from $\R^n$ to $\R$.}

Given a $C^{k,\a}$ convex body $D$ in $\R^n$ with $k\ge 1$, let $\nu_D\co \d D\to S^{n-1}$ 
be the {\em Gauss map} given by the outward unit normal. 
For $k\ge 2$ the {\em Gaussian curvature} 
is the determinant of the differential of $\nu_D$.
It is well-known that the Gaussian curvature of a convex body is nonnegative.
Note that $\nu_D$ is $C^{k-1, \a}$. Also 
$\nu_D$ is a $C^1$ diffeomorphism if and only if the Gaussian curvature is positive, and
$\nu_D$ is a homeomorphism if and only if
$D$ is {\em strictly convex\,}, i.e., $\d D$ contains no line segments.
Positive Gaussian curvature implies strict convexity.


A convex body $D$ is strictly convex if and only if $h_D$ is differentiable away from $o$, 
and furthermore, if $D$ is strictly convex, then the restriction of 
$\nabla h_D$ to $S^{n-1}$ equals $\nu_D^{-1}$ so that $h_D$ is $C^1$, 
see~\cite[Corollary 1.7.3]{Sch-book}. 

\begin{lem} 
Let  $k\ge 2$ and $\a\in\II$. For $D\in \K(\R^n)$ the following are equivalent
\begin{itemize}
\item[\textup{(1)}] $D$ is a $C^{k,\a}$ convex body and $\d D$ has positive Gaussian curvature,
\item[\textup{(2)}]
$h_D\vert_{S^{n-1}}$ is $C^{k,\a}$ and
$\nabla h_D\vert_{S^{n-1}}$ has no critical points.
\end{itemize}
\end{lem}
\begin{proof}
Let us show $\mathrm{(1)}\implies\mathrm{(2)}$. 
Since $\d D$ is $C^{k,\a}$, the Gauss map 
$\nu_D$ is $C^{k-1,\a}$. 
Nonvanishing of the Gaussian curvature of $\d D$ means that
$\nu_D$ is a $C^1$ diffeomorphism, and hence a $C^{k-1,\a}$ diffeomorphism
by the inverse function theorem, see~\cite[Theorem 2.1]{BHS-holder}. 
Now $\nabla h_D\vert_{S^{n-1}}=\nu_D^{-1}$, 
implies that
$h_D\vert_{S^{n-1}}$ is $C^{k,\a}$ and
$\nabla h_D\vert_{S^{n-1}}$ has no critical points. 

To show $\mathrm{(2)}\implies\mathrm{(1)}$ first note that the assumption
$h_D\vert_{S^{n-1}}$ is $C^{k,\a}$ and the homogeneity
of $h_D$ shows that $h_D$ is $C^{k,\a}$ on $\R^n\setminus\{o\}$, and in particular, is
differentiable there. The latter implies that every support hyperplane meets $D$ in
precisely one point, see~\cite[Corollary 1.7.3]{Sch-book}, and in particular $D$ has 
nonempty interior and $\nu_D$ is a homeomorphism. As was mentioned above,
$\nabla h_D\vert_{S^{n-1}}=\nu_D^{-1}$, so
using the assumptions we conclude that $\nu_D^{-1}$ is a $C^{k-1,\a}$ 
diffeomorphism and $\d D$ is a $C^{k-1,\a}$
submanifold. These two statements in fact imply that $\d D$ is $C^{k,\a}$, see 
e.g., the proof of~\cite[Lemma 5.4]{Gho}. Finally, non-vanishing of the Gaussian curvature
is equivalent to $\nu_D$ being an immersion.
\end{proof}

The set of $C^{k,\a}$ convex bodies of positive Gaussian curvature is convex
under scaling and Minkowski addition, see~\cite[Proposition 5.1]{Gho} for $\a=0$,
and~\cite{BelJia} in general.

We make heavy use of 
the map $\sf\co\K(\R^n)\to C({\mathbb S}^{n-1})$ given by $\sf (D)=h_D\vert_{{\mathbb S}^{n-1}}$ which
enjoys the following properties:
\begin{itemize}
\item
$\sf$ is an isometry onto its image, where as usual the domain has 
the Hausdorff metric and the co-domain has the metric induced by
the $C^0$ norm~\cite[Lemma 1.8.14]{Sch-book}.
\item
the image of $\sf$ is closed~\cite[Theorem 1.8.15]{Sch-book} and convex~\cite[pp.45 and 48]{Sch-book}
in $C({\mathbb S}^{n-1})$.
\item
$\sf $ is Minkowski linear~\cite[Section 3.3]{Sch-book}, 
i.e., $\sf(a\,D+b\,K)=a\sf(D)+b\sf(K)$ for any nonnegative $a$, $b$ and any $D,K\in\K$.
\end{itemize}

The {\em Steiner point} is a map $s\co\mathcal{K}(\R^n)\to\R^n$   
given by \[
s(D)=\frac{1}{\vol(\mathbb B^n)}\int_{S^{n-1}} u\, h_D(u)\, du
\]
which has the following properties:
\begin{itemize}\vspace{-2pt}
\item 
$s$ is Lipschitz~\cite[p.66, Section 1.8]{Sch-book},\vspace{2pt}
\item
$s$ is invariant under rigid motions, i.e., $s(gD)=gs(D)$ for any 
$g\in\Iso(\R^n)$~\cite[p.50, Section 1.7]{Sch-book},\vspace{2pt}
\item
$s$ is Minkowski linear, i.e., $s(aD+bK)=as(D)+bs(K)$ for any positive reals 
$a,b$ and $D, K\in\mathcal K(\R^n)$. \vspace{2pt}
\item
$s(D)$ lies in the relative interior of $D$~\cite[p.315, Section 5.2.1]{Sch-book},\vspace{2pt}
\item
if $D$ is a point, then $s(D)=D$~\cite[p.50, Section 1.7]{Sch-book},\vspace{2pt}
\item
$s$ is the only continuous  Minkowski linear, $\Iso(\R^n)$-invariant 
map from $\mathcal K(\R^n)$ to $\R^n$~\cite[Theorem 3.3.3]{Sch-book}. 
\end{itemize}
Thus the hyperspace $\K_s$ of convex compacta in $\R^n$ with Steiner point at $o$
is an  $O(n)$-invariant closed convex subset of 
$\mathcal K$ and the map 
\begin{equation}
\label{form: st pt times Rn}
\mathcal K\to\R^n\times\K_s
\end{equation}
sending $D$ to $(s(D), D-s(D))$ is a homeomorphism. 

\begin{lem}
\label{lem: st point}
The retraction ${\mathbf r}\co\mathcal K\to\K_s$ given by ${\mathbf r}(D)=D-s(D)$ 
is equivariant under the homomorphism $\mathrm{Iso}(\R^n)\to O(n)$ 
given by $\phi\to\phi-\phi(0)$, and descends to a
homeomorphism $\bar{\mathbf r}\co\mathcal K/\mathrm{Iso}(\R^n)\to\K_s/O(n)$.
\end{lem}
\begin{proof}
The equivariance of ${\mathbf r}$ and bijectivity of $\bar{\mathbf r}$ is straightforward
from the properties of the Steiner point and the fact that any
isometry of $\R^n$ can be written as $x\to Ax+b$ for some unique 
$A\in O(n)$ and $b\in \R^n$.
Bijectivity of $\bar{\mathbf r}$ implies that 
the inclusion $i\co\mathcal \K_s\to\K$ descends to $\bar{\mathbf r}^{\,-1}$, 
and since ${\mathbf r}$, $i$ are continuous, so are $\bar{\mathbf r}$, $\bar {\mathbf r}^{\,-1}$.
\end{proof}

\begin{lem}
\label{lem: add ball}
If $t>0$ and $D\in\K_s$ with $h_D\vert_{_{{\mathbb S}^{n-1}}}\in C^\infty$, 
then $D+t\,\mathbb B^n\in \B_p$.
\end{lem}
\begin{proof} Set $B=t\,\mathbb B^n$.
From $h_{D+B}=h_D+h_B$ we conclude that $h_{D+B}$ is $C^\infty$.
The equality
$\nu_{_{D+B}}^{-1}=\nabla h_{_{D+B}}\vert_{S^{n-1}}$
implies that $\nu_{_{D+B}}^{-1}$ is $C^\infty$.
To show that $D+B$ is in $\B_p$
let us check that $\nu_{_{D+B}}$ is a $C^\infty$ diffeomorphism,
which by the inverse function theorem is equivalent to 
$\nu_{_{D+B}}^{-1}$ having no critical points, i.e., that 
\[
\mathrm{Hess}\, h_{_{D+B}}=
\mathrm{Hess}\, h_{_{D}} + \mathrm{Hess}\, h_{_{B}}\vspace{2pt}
\] 
is positive definite. The Hessian of any convex $C^2$ function is
nonnegative definite which applies to $\mathrm{Hess}\ h_{_{D}}$
while $\mathrm{Hess}\, h_{_{B}}$ is positive definite
e.g., because $\nabla h_{_{t\,{\mathbb B}^n}}(x)=t\,x/\|x\|$
is a diffeomorphism on $\mathbb S^{n-1}$. 
\end{proof}

A key tool in this paper is the {\em Schneider's regularization\,} which is immediate from
the remark after the proof of~\cite[Theorem 3.4.1]{Sch-book}. 

\begin{lem}
\label{lem: regularization} 
There is a continuous map $\r\co\K_s\times\II\to\K_s$ 
such that $\r(\cdot,0)$ is the identity map, and  for each $t>0$
the map $\r(\cdot, t)$ is $O(n)$-equivariant and has image in $\B_p$.
\end{lem}
\begin{proof}
Fix a nonnegative function 
$\psi\in C^\infty(\R)$
with support in $[1,2]$ and such that $\int_\R \psi=1$.
For $t\in (0,1]$ and $D\in \K$ let $\r(D)=T(D)+t\,{\mathbb B}^n$
where $T(D)$ is the convex set with support function
\[
h_{_{T(D)}}(x)=\int_{\R^n} h_D(x+z\|x\|)\,\psi(\|z\|/t)\, t^n\, dz.
\]
Note that $h_{_{\r(D)}}=h_{_{T(D)}}+h_{_{t\,{\mathbb B}^n}}$
and $h_{_{t\,{\mathbb B}^n}}(x)=t\,\|x\|$.
It follows from~\cite[Theorem 3.4.1]{Sch-book} that 
$h_{_{T(D)}}$ is $C^\infty$ on $\R^n\setminus\{o\}$
and $T$ is $O(n)$-equivariant and continuous on $\K_s\times\II$.
These properties are clearly inherited by $\r$.
The image of $\rho$ is in $\B_p$ by Lemma~\ref{lem: add ball}.
\end{proof}

Another useful operation is what we call
{\em $(\e, u)$-truncating\,} of a $C^1$ convex body $K$,
where $\e\in\II$ and $u\in\mathbb S^{n-1}$, defined
as the convex set $K_{u,\e}$ obtained
by removing all points of $K$
that lie in the open $\e$-neighborhood of the support hyperplane to $K$
with the normal vector $u$, and then subtracting the Steiner point
of the result. For each $u$ the map $\K_s\times \II\to\K_s$ 
that sends $(K, \e)$ to $K_{u,\e}$ is continuous.
If $\e$ is smaller than the length of the projection
of $K$ onto the line $\mathrm{span}\{u\}$, then $K_{u,\e}$
is not $C^1$.

\section{The hyperspace of $\R^n$ of convex compacta of dimension $\ge 2$}
\label{sec: hyperspace of dim >1}

In this section we determine the homeomorphism types of $\K_s$ and $\K_s^{2\le n}$
which is straightforward 
but apparently not in the literature.
Our interest in $\K_s^{2\le n}$ stems from the fact that
$\bm{\K}_s^{2\le 3}(\R^3)$ can be identified with the Gromov-Hausdorff
space of convex surfaces, see~\cite{Bel-ghs2}.

\begin{lem}
$\K_s$ is homeomorphic to $Q\times [0,1)$.
\end{lem}
\begin{proof}
By~\cite{NQS} the hyperspace $\K$ is homeomorphic to a once-punctured Hilbert cube,
which in turn is homeomorphic to $Q\times [0,1)$, see~\cite[p.118]{BP-book}
or~\cite[Theorem 25.1]{Cha-book}.
Since $\R^n\times\K_s$ is homeomorphic to $\K$,
we conclude that $\K_s$ is infinite dimensional and locally compact.
Since $\K$ is an AR that retracts onto $\K_s$, the latter is an AR.

Properties of $\sf$ reviewed in Section~\ref{sec: conv geom} imply that
$\K_s$ is homeomorphic to $\sf(\K_s)$ which is a closed convex subset
of $C(\mathbb S^{n-1})$. 
Now~\cite[Theorem 5.1.3]{BRZ-book} classifies 
closed convex subsets in linear metric spaces
that are infinite-dimensional locally compact absolute retracts
as spaces homeomorphic to $Q\times [0,1)$ and $Q\times\R^k$, $k\ge 0$.
These spaces are pairwise
non-homeomorphic, see~\cite[p.116, Theorem III.7.1]{BP-book} or~\cite[Theorem 25.1]{Cha-book}.
Since $\R^n\times\K_s$ is homeomorphic to $\K$,
which in turn is homeomorphic to $Q\times [0,1)$, the hyperspace $\K_s$
cannot be homeomorphic to $Q\times\R^k$, and hence it must be
homeomorphic to $Q\times [0,1)$.
\end{proof}

The homeomorphism  $\K\to\R^n\times\K_s$ given by (\ref{form: st pt times Rn})
is dimension preserving, so it restricts to the homeomorphism 
$\K^{2\le n}\to\R^n\times\K_s^{2\le n}$. 

\begin{lem}
\label{lem: Z-set dim at most 1}
\begin{enumerate}
\item[\textup{(1)}]
$\K_s^{0\le 1}$ is a $Z$-set in $\K_s$.\vspace{2pt}
\item[\textup{(2)}] 
The one-point compactification of $\K_s^{0\le 1}$ is a $Z$-set in 
the one-point compactification of $\K_s$ which is is homeomorphic to $Q$. \vspace{2pt}
\item[\textup{(3)}]
$\K_s^{0\le 1}$  homeomorphic to 
$\R^n/\{\pm 1\}$, the open cone over $RP^{n-1}$.\vspace{3pt}
\item[\textup{(4)}] 
$\K_s^{2\le n}$ is a contractible $Q$-manifold which is obtained from $Q$
be deleting a $Z$-set homeomorphic to the suspension over $RP^{n-1}$.
\end{enumerate}
\end{lem}
\begin{proof}
Note that $\K_s^{0\le 1}$ consists of line segments (of possibly zero length)
with Steiner point at the origin. Hence $\K_s^{0\le 1}$
is homeomorphic to the quotient space $\R^n/\{\pm 1\}$.
Also $\K_s^{0\le 1}$ is  $Z$-set 
because any map $f\co Q\to\K(\R^n)$ 
can be approximated by $f_\e$, where $f_\e(q)$ is the $\e$-neighborhood of $f(q)$,
which clearly has dimension $n\ge 2$.
As $\K_s^{0\le 1}$ is closed and noncompact in $\K_s$, the 
one-point compactification of $\K_s^{0\le 1}$
embeds into the one-point compactification of $\K_s$, into
which $\K^{2\le n}$ projects homeomorphically to $Q$ because the inclusion
of a space into its one-point compactification 
is an open map. 
Since $\K_s$ is homeomorphic to
the once-punctured copy of $Q$, its one-point compactification is homeomorphic to $Q$. 
Any point of $Q$ is a $Z$-set, so
the one-point compactification takes any $Z$-set in $\K_s$
to a $Z$-set in $Q$ (because a map $Q\to Q$ can be first pushed off the
added point, and then pushed off the $Z$-set inside $\K_s$). 
The one-point compactification of $\K_s^{0\le 1}$
is a $Z$-set in $Q$
homeomorphic to $SRP^{n-1}$, the suspension over the real projective $(n-1)$-space.
Thus $\K^{2\le n}$ is homeomorphic to the complement
in $Q$ of a $Z$-set homeomorphic to $SRP^{n-1}$.
\end{proof}

\begin{rmk}
Any two homeomorphic $Z$-sets in $Q$ are ambiently homeomorphic~\cite[Theorem 11.1]{Cha-book}, and
in particular, the part (4) of Lemma~\ref{lem: Z-set dim at most 1}
uniquely describes $\K_s^{2\le n}$ up to homeomorphism.
Note that the one-point compactification of $\K_s^{0\le 1}$
can be moved by an ambient homeomorphism to a face of $Q$
because any closed subset of a face is clearly a $Z$-set.
\end{rmk}

Chapman showed~\cite[Theorem 21.2]{Cha-book} that homotopy equivalent
$Q$-manifolds become homeomorphic after multiplying by $[0,1)$, and hence
for the products of $Q$-manifolds with $[0,1)$ 
their homeomorphism and homotopy classifications coincide.
A commonly used unpublished result of Wong, see~\cite[p.275]{Cur-peano}, says that a $Q$-manifold $M$
is homeomorphic to $M\times [0,1)$ if and only if for each compact subset $A$ in $M$ there is a proper
homotopy $f_t\co M\to M$ such that $f_1$ is the identity and $f_t(M)\subset M\setminus A$ for some $t>1$.  

\begin{lem} 
\label{lem: prod with [0,1)}
If $M$ is 
$\accentset{\circ}{\K}^{k\le l}_s$ or $\accentset{\circ}{\bm{\K}}^{k\le l}_s$,
then $M$ is homeomorphic to $M\times [0,1)$.
\end{lem}
\begin{proof}
Suppose $M=\accentset{\circ}{\K}^{k\le l}_s$ and $f_t(K)=tK$
where $K\in M$ and $t\ge 1$. The diameter of $tK$ is $t$ times the diameter of $K$.
Since $o\notin M$, 
any $K\in M$ has positive diameter, which is bounded above on any 
compact subset $A$ of $M$. Thus $f_T(M)$ is disjoint from $A$ for some $T>1$.
The map $(K,t)\to tK$ is proper on $M\times [0,T]$.
Hence by Wong's result $M$ is homeomorphic to $M\times [0,1)$.  
The same  holds for $M=\accentset{\circ}{\bm{\K}}^{k\le l}_s$ because
scaling commutes with the $O(n)$-action, so that $f_t$ descends to a homotopy 
of the orbit spaces that eventually pushes $M$ off any compact subset. 
\end{proof}

Let $\K_s(\mathbb B^n)$ be the hyperspace consisting of sets in $\K_s$
contained in $\mathbb B^n$. It is compact by the Blaschke selection theorem.

\begin{lem}
\label{lem: Ks(Bn)}
$\K_s(\mathbb B^n)$ is homeomorphic to $Q$ and $\K_s^{2\le n}(\mathbb B^n)$ 
is homeomorphic to $Q\times [0,1)$.
\end{lem}
\begin{proof}
Note that $\sf$ maps $\K_s(\mathbb B^n)$ onto a compact convex 
subset of $C(\mathbb S^{n-1})$ which is infinite-dimensional,
which can be seen, e.g., by embedding $Q$ into $\K_s$ and then rescaling to embed it
into $\K_s(\mathbb B^n)$. Any compact convex infinite-dimensional subset of a Banach space
is homeomorphic to $Q$~\cite[p.116, Theorem III.7.1]{BP-book},
so $\K_s(\mathbb B^n)$ is homeomorphic to $Q$.
Now $\K_s^{0\le 1}(\mathbb B^n)$ is homeomorphic to the cone over $RP^{n-1}$,
and in particular it is contractible, and hence has the shape of a point. 
It is also a $Z$-set in $\K_s(\mathbb B^n)$: 
contract the image of $Q\to \K_s(\mathbb B^n)$ by rescaling and then take
$\e$-neighborhood to push it off $\K_s^{0\le 1}(\mathbb B^n)$. Two $Z$-sets in $Q$ the same shape
if and only of their complement are homeomorphic~\cite[Theorem 25.1]{Cha-book},
so $\K_s^{0\le 1}(\mathbb B^n)$ is homeomorphic to a once-punctured Hilbert cube, 
and hence to $Q\times [0,1)$.
\end{proof}

\section{On hyperspaces homeomorphic to $\Si^\o$}
\label{sec: hyp homeo to Si-to-o}

Let us try to isolate the conditions
on a hyperspace that would make it homeomorphic to $\Si^\o$.
Throughout the section we fix $\a\in\II$ and let 
$\mathcal D$ denote an arbitrary 
hyperspace of $\R^n$ satisfying $\B_p\subset\mathcal D\subset\B^{1,\a}$.

\begin{lem}
\label{lem: between set is homotopy dense AR}
If $\B_p\subset X\subset\K_s$, then
$X$ is an AR that is homotopy dense in $\K_s$.
This applies to $X=\mathcal D$.
\end{lem}
\begin{proof}
By Schneider's regularization $\B_p$ is dense in $\K_s$. The map $\sf$
homeomorphically takes $\K_s$ and $\B_p$ to convex subsets of
$C(\mathbb S^{n-1})$.
By~\cite[Exercises 12 and 13 in Section 1.2]{BRZ-book} any dense convex subset of a set
in a linear metric space is homotopy dense. 
Thus $\B_p$ is homotopy dense in  $\K_s$, and hence so is $X$.
Any homotopy dense subset of an AR is an AR~\cite[Proposition 1.2.1]{BRZ-book}.
This applies to $X$ because $\K_s$ is an AR; in fact, 
$\K_s$ is homeomorphic to a once-punctured Hilbert cube~\cite{NQS}.
\end{proof}

\begin{lem}
\label{lem: between set has SDAP}
$\mathcal D$ has SDAP.
\end{lem}
\begin{proof}
Lemma~\ref{lem: between set is homotopy dense AR} shows
homotopy density of $\B^{1,\a}$  in $\K_s$.
Also $\B^{1,\a}$ is homotopy negligible in $\K_s$
because if we fix $u\in\mathbb S^{n-1}$ and a map $Q\to\K_s$, 
then there is $\e>0$ such that for any $D\in f(Q)$
the result of $(\e, u)$-truncating of $D$
is not $C^1$.
Since $\K_s$ is locally compact, it has LCAP, 
and~\cite[Exercise 12h in Section 1.3]{BRZ-book} implies 
that any homotopy dense and homotopy negligible subset of an ANR with LCAP
has SDAP. Thus $\B^{1,\a}$ has SDAP.
By~\cite[Exercise 4 in Section 1.3]{BRZ-book} every homotopy dense
subset of an ANR with SDAP has SDAP, and this applies to $\mathcal D$.
\end{proof}

\begin{lem} 
\label{lem: sZ}
If $\a>0$, then $\mathcal D$ is $\s Z$, and it lies in a $\s Z$-subset of $\K_s$. 
\end{lem}
\begin{proof}
For a convex body $K$ in $\R^n$ consider its orthogonal projection $\bar K$ 
to the hyperplane $\{x\in\R^n\,:\, x_n=0\}$. Let $s(\bar K)\in\mathrm{Int}(\bar K)$ 
be the Steiner point of $\bar K$, consider the largest ball about $s(\bar K)$ that is contained
in $\bar K$, and let $B_K$ be the ball half that radius about $s(\bar K)$.
Consider the portion of $\d K$ that is the graph of a convex
function on $B_K$ and precompose the function with the map $\mathbb B^{n-1}\to B_K$
that is the composition of a dilation followed by the translation by $s(\bar K)$.
The result is a convex function by $f_K\co \mathbb B^n\to\R$.
It is easy to see that the map $\K^n\to C(\mathbb B^{n-1})$
given by $K\to f_K$ is continuous.

For $m\in\o$ let $\L_m$ be the set of functions $f\in C(\mathbb B^n)$ such that  for every $r\in (0,\a)$
the $C^{1, r}$ norm of $f$ is at most $m$. Equip $\L_m$ with the $C^0$ topology.
A version of the Arzel{\`a}-Ascoli theorem, see~\cite[Lemma 6.36]{GilTru}, 
implies that $\L_m$ is compact.

Let $\hat Z_m=\{K\in\K_s\,:\, f_K\in\L_m\}$ and $Z_m=\hat Z_m\cap\mathcal D$.
Thus $\hat Z_m$, $Z_m$ are closed in $\K_s$, $\mathcal D$, respectively.

The equality $\mathcal D=\bigcup_m Z_m$ follows from the facts that
for each $\r\in (0,r)$ any $C^{1,\a}$ function on $D$ has finite $C^{1,\r}$ norm, and 
the Lipschitz constant of the identity map of $C^{1,\a}(D)$,
where the domain and the co-domain are respectively given the $C^{1,r}$ and $C^{1,\r}$ norms,
is bounded above independently of $\r$, see~\cite[Lemma 6.35]{GilTru}. 

To show that $Z_m$ is a $Z$-set
we start from a continuous map $f\co Q\to \mathcal D$ and 
try to push it off $Z_m$ inside $\mathcal D$.
Let $n_{_K}$ be the outward normal vector to the graph of $f_K$ 
at the point that projects to $s(\bar K)$. A basic property of $C^1$ convex 
bodies is that $n_{_K}$ varies continuously with $K$.
Apply $(\e, n_{_K})$-truncating to $f$, and then Schneider's $\tau$-regularization.

Since $Q$ is compact for all sufficiently small $\e$ the result of $(\e, n_{_K})$-truncating
of each body in $f(Q)$ is not $C^1$. For small $\tau$ the result of the above procedure 
will have very large $C^1$ norm, and hence it will not intersect $Z_m$.
(If it did, then for some $r>0$ the $C^{1,r}$ norm would be bounded uniformly in $\de$,
and the Arzel{\`a}-Ascoli theorem would give a subsequence converging in the
$C^1$ norm, but the limit is not $C^1$).

To show that $\hat Z_m$ is a $Z$-set in $\K_s$ start 
from a continuous map $f\co Q\to \mathcal \K_s$, push it into
$\mathcal D$ by Schneider's regularization, and then push it off $Z_m$
inside $\mathcal D$ as above. Since $Z_m=\hat Z_m\cap\mathcal D$,
the resulting map will miss $\hat Z_m$.
\end{proof}

Lemma~\ref{lem: sZ} seems to be false for $\a=0$ but we have no use for 
this assertion hence we will not attempt to justify it.

\begin{rmk}
\label{rmk: invariant Z-set}
If $\mathcal D$ in Lemma~\ref{lem: sZ} is $O(n)$-invariant,
then $I_m=\bigcap_{g\in O(n)}g(Z_m)$ is a closed subset of $Z_m$,
and hence a $Z$-set in $\mathcal D$. The facts that $O(n)$ is compact
and the $C^{1,\a}$ norm of $f_K$ varies continuously under slight rotations
of the  graph of $f_K$ easily imply that $\mathcal D=\bigcup_{m\in\o} I_m$. 
Similarly, $\hat I_m=\bigcap_{g\in O(n)}g(\hat Z_m)$ is an $O(n)$-invariant
$Z$-set in $\K_s$, and $\{I_m\}_{m\in\o}$ covers 
$\mathcal D$ because $\hat I_m\supset I_m$.
\end{rmk}

\begin{lem}
\label{lem: B_p is M_2}
$\B_p$ belongs to $\mathcal M_2$.
\end{lem} 
\begin{proof}
Let $\B_p^\infty$ denote the set $\B_p$ with the $C^\infty$ topology.
In this topology the Gaussian curvature of any convex body in $\B_p^\infty$ varies 
continuously. 
Thus $\B_p^\infty$ is precisely the subset of hypersurfaces
of positive Gaussian curvature in the space of all compact $C^\infty$ hypersurfaces
in $\R^n$ equipped with the $C^\infty$ topology. The latter space 
is Polish, see~\cite{GV}. Any open subset
of a Polish space is Polish, hence 
$\B_p^\infty$ is Polish.

For $\g\in\{0,\infty\}$ let $\Gamma^\g(\mathbb S^{n-1})$ denote
the set $C^\infty(\mathbb S^{n-1})$ equipped with the $C^\g$
topology. Let $\sf^\infty\co\B_p^\infty\to \G^\infty(\mathbb S^{n-1})$ be the map 
that associates to a convex body its support function, i.e., $\sf^\infty=\sf$ as maps of sets.
Similarly to $\sf$, the map $\sf^\infty$ is a topological embedding because the support function
for sets in $\B_p$ equals the distance to $o$ from the support hyperplane, 
and both the tangent plane and the distance to $o$ vary in the $C^\infty$ topology
as $o$ lies in the interior of each set in $\B_p$.
Since $\B_p^\infty$ is Polish, its
homeomorphic $\sf^\infty$-image is $G_\delta$. By~\cite[Lemma 5.2]{BanBel}
the identity map $\Gamma^\infty(\mathbb S^{n-1})\to \Gamma^0(\mathbb S^{n-1})$
takes any $G_\delta$ subset to a space in $\mathcal M_2$. Thus
$\B_p$ is in $\M_2$.
\end{proof}

\begin{lem}
\label{lem: M_2}
$\mathcal D$ is in $\M_2$
if $\mathcal D\setminus\B_p$ lies in a subset of $\mathcal D$
that belongs to $\M_2$.
\end{lem}
\begin{proof}
$\B_p$ is in $\M_2$ by Lemma~\ref{lem: B_p is M_2}.
Hence $\mathcal D$ is the union of two subsets that belong to $\M_2$,
the class of absolute $F_{\s\de}$ sets, i.e., their homeomorphic images
in any metric space are $F_{\s\de}$, and in particular, this is true in $\K_s$. 
The union of two $F_{\s\de}$ subset is $F_{\s\de}$, so $D$ is a $F_{\s\de}$
in $\K_s$, which is complete and therefore $\mathcal D$ 
is in $\M_2$~\cite[Theorem 1.1, p.266]{BP-book}.
\end{proof}

Lemma~\ref{lem: B_p is in M2-universal} below depends on
the following theorem proved in~\cite[Theorem 5.1 and Corollary 4.9]{BanBel}.

\begin{thm}
\label{thm: criterion M_2-univer}
Let $N$ be a smooth manifold, possibly with boundary, and 
let $D\subset \mathrm{Int}\, N$ be a smoothly embedded top-dimensional closed disk
that is mapped via a coordinate chart to a Euclidean unit disk. Let $l\ge 0$ be an integer and 
let $\mathfrak D\co C^l(N)\to C(N)$ be a continuous linear map.
Given $\eta\in\R$ suppose there exists $h_\bu\in C^\infty(N)$ with 
$\mathfrak D h_\bu\vert_{_D}>\eta$.
Let $C^l_\bu$ denote the subspace of $C^l(N)$ of functions $u$ 
such that $\mathfrak D u\vert_D\ge \eta$ and $u\vert_{N\setminus \mathrm{Int}\,D}=h_\bu$. 
Let $C^n_\bu=C^l_\bu\cap C^n(N)$ and $C^\infty_\bu=C^l_\bu\cap C^\infty(N)$. Let $k\ge l$.
If $f\co C^k_\bu\to X$ is a continuous injective map to a Hausdorff topological space $X$,
then the subspace $f(C^\infty_\bu)$ of $X$ is $\M_2$-universal.
\end{thm} 

\begin{lem}
\label{lem: B_p is in M2-universal}
$\B_p$ is strongly $\mathcal M_2$-universal.
\end{lem}
\begin{proof}
Let ${\hat \B}_p$ denote
the hyperspace of all positively curved $C^\infty$ convex bodies in $\R^n$.
The map (\ref{form: st pt times Rn}) restricts to a homeomorphism ${\hat \B}_p\to\R^n\times\B_p$.
Thus the products of ${\hat \B}_p$ and $\B_p$ with $\R^\o$ are homeomorphic.

Lemmas~\ref{lem: between set is homotopy dense AR}--\ref{lem: between set has SDAP} 
show that $\B_p$ is an AR with SDAP, and hence so is ${\hat \B}_p$. 
By~\cite[Theorem 3.2.18]{BRZ-book} if $X$ is an ANR with SDAP, then $X$ is strongly 
$\mathcal M_2$-universal if and only if $X\times\R^\o$ is strongly $\mathcal M_2$-universal.
Thus it suffices to show that ${\hat \B}_p$ is strongly $\mathcal M_2$-universal.

By~\cite[Proposition 5.3.5]{BRZ-book} a convex AR with SDAP in a linear metric space $L$
is strongly $\M_2$-universal if it contains an $\mathcal M_2$-universal subset
that is closed in $L$. Set $L=C^\infty(\mathbb S^{n-1})$ with the $C^0$ norm.
Recall that $\sf$ maps $\K$ homeomorphically onto a convex subset of $L$.
By~\cite[Proposition 5.1]{Gho} $\sf({\hat \B}_p)$ is convex in $L$,
see also~\cite[Theorem 1.1]{BelJia}.

To find an $\mathcal M_2$-universal subset of ${\hat\B}_p$ we 
apply Theorem~\ref{thm: criterion M_2-univer} 
to $N=[0,1]$, $D=\left[\frac13,\frac12\right]$, $k=l=2$, $h_\bullet(r)=r^2$, $\eta=1$, 
and $\mathfrak D (h)=h^{\prime\prime}$.
Thus 
\[
C^2_\bullet=\left\{ 
u\in  C^2(\II)\,:\, u^{\prime\prime}\ge 1\ 
\text{and $u(r)=r^2$ for all $r\notin \left[\frac13,\frac12\right]$}
\right\}
\]
and $C^\infty_\bullet=C^2_\bullet\cap C^\infty(\II)$.
Define a map $f\co C^2_\bullet\to {\hat\B}_p$
as follows. 
Fix $A\in {\hat\B}_p$ that is invariant under $SO(n-1)$ rotations
about the $x_n$-axis and such that the portion of $\d A$ satisfying
$x_n\in [0,1]$ equals the paraboloid 
\[
\{x=(x_1,\dots , x_n)\in \R^n : x_n=r(x)^2\  \text{and}\ x_n\in [0,1]\}
\] 
where $r(x)$ is the distance from $x$ to the $x_n$-axis in $\R^n$.
Let $f$ send an element $q\in C^2_\bullet$ to the convex body $\mathbb B^n+A_q$
where $A_q$ obtained from $A$ by replacing 
$x_n=r^2$ with $x_n=q(r)$ in the above paraboloid portion of $\d A$.
One checks that $A_q$ has positive curvature, so that
both $A_q$ and $\mathbb B^n+A_q$ lie in ${\hat\B}_p$.
It is easy to see that $\sf\circ f$ is a homeomorphism onto its image. 
By Theorem~\ref{thm: criterion M_2-univer} 
$(\sf\circ f)(C^\infty_\bullet)$ is $\M_2$-universal.

It remains to show that $(\sf\circ f)(C^\infty_\bullet)$ is
closed in $L$.  Note that $f(C^\infty_\bullet)$ lies in a
compact subset of $\K_s$, which is therefore mapped by
$\sf$ homeomorphically onto its compact image.
Thus any limit point of $(\sf\circ f)(C^\infty_\bullet)$
lies in $\sf(\K_s)$ and it is enough to show that  
$f(C^\infty_\bu)$ in closed in the hyperspace consisting of convex bodies in $\R^n$
with $C^\infty$ support functions. Fix a sequence $\mathbb B^n+A_{q_m}$ in 
$f(C^\infty_\bu)$ that converges to a convex body with $C^\infty$ support function.
By construction the limit is of the form $\mathbb B^n+A_{\k}$
where $q_m\in  C^\infty_\bu$ converge in the uniform $C^0$ topology
to a (necessarily convex) function $\k$. By assumption $\sf(\mathbb B^n +A_\k)$ is $C^\infty$,
and hence so is $\sf(A_{\k})$. Now Lemma~\ref{lem: add ball} 
implies that $\mathbb B^n +A_\k\in{\hat\B}_p$. Hence $A_\k\in{\hat\B}_p$~\cite[Theorem 1.1]{BelJia}
and therefore $\k$ is $C^\infty$.
Since $q_m(r)=r^2$ for 
$r\notin \left[\frac13,\frac12\right]$, the same is true for $\k$. 
Set $p(r)=\frac{r^2}{2}$.
Since $q_m^{\prime\prime}\ge 1$, each function
$q_m-p$ is convex, and hence so is 
$\k-p=\displaystyle{\lim_{\scriptscriptstyle{m\to\infty}} q_m-p}$.
Since $\k$ is $C^\infty$, we get $(\k-p)^{\prime\prime}\ge 0$ so 
$\k^{\prime\prime}\ge 1$, and hence 
$\k\in C^\infty_\bu$ as claimed. 
\end{proof}

\begin{lem}
\label{lem: enlarge}
If $\B_p$ is $G_\de$ in $\mathcal D$, then
$\mathcal D$ is  strongly $\mathcal M_2$-universal.
\end{lem}
\begin{proof}
The~\cite[Enlarging Theorem 3.1.5]{BRZ-book}
implies that an ANR with SDAP is strongly $\M_2$-universal if and only if it contains
a  strongly $\mathcal M_2$-universal homotopy dense $G_\de$ subset.
By assumption $\B_p$ is $G_\de$ in $\mathcal D$.
The space $\B_p$ is strongly $\M_2$-universal 
due to Lemma~\ref{lem: B_p is in M2-universal}.
By Lemmas~\ref{lem: between set is homotopy dense AR}--\ref{lem: between set has SDAP} 
the space $\mathcal D$ is an ANR with SDAP and $\B_p$ is homotopy dense in $X$. 
\end{proof}

\begin{thm}
\label{thm: D homeo to sigma to the omega}
If $\a>0$ and $\mathcal D\setminus \B_p$ is $\s$-compact, then
$\mathcal D$ is homeomorphic to $\Si^\o$, and 
in particular, $\B_p$  is homeomorphic to $\Si^\o$. 
\end{thm}
\begin{proof}
Note that any $\s$-compact subspace is both $F_\s$
and in $\M_2$. By the above lemmas $\mathcal D$ is an AR with SDAP, $\s Z$, strongly $\M_2$-universal, and
is in $\M_2$, so that $\mathcal D$ is $\M_2$-absorbing, and the only such space is $\Si^\o$.
\end{proof}

Can $\s$-compactness of $\mathcal D\setminus \B_p$ 
in Theorem~\ref{thm: D homeo to sigma to the omega} can be replaced by a 
weaker condition that holds for examples of interest such as
$\B^\infty$? The following result illustrates what could go wrong.  

\begin{thm}
\label{thm: ex of D not in M_2}
There is a hyperspace $\mathcal D$ with $\B_p\subset\mathcal D\subset\B^{1,\a}$ 
such that $\mathcal D\setminus\B_p$ embeds into the Cantor set, 
$\B_p$ is open in $\mathcal D$, and $\mathcal D$ is not a topologically
homogeneous, and in particularly, not a $\Si^\o$-manifold.
\end{thm}
\begin{proof}
For any uncountable Polish space, such as the Cantor set, the Borel hierarchy of its
subsets does not stabilize~\cite[Theorem 22.4]{Kec-book-set-th}, 
so in particular, it contains a subset not in $\M_2$.
Use Lemma~\ref{lem: any space embeds} below to embed it
onto a subset $\Lambda$ of $\B^{1,\a}\setminus\B_p$
that is closed in $\Lambda\cup\B_p$. Since 
$\M_2$ is closed-hereditary, $\Lambda\cup\B_p$ is not in $\M_2$ and hence
not a $\Si^\o$-manifold. If $\Lambda\cup\B_p$ were topologically homogeneous,
then it would be a $\Si^\o$-manifold because 
the $\Si^\o$-manifold $\B_p$ is open in $\Lambda\cup\B_p$.
\end{proof}

The earlier results in this section imply that $\mathcal D$ in Theorem~\ref{thm: ex of D not in M_2}
is a strongly $\M_2$-universal ANR with SDAP which is also $\s Z$ if $\a>0$.

\begin{lem}
\label{lem: any space embeds}
Any 
space is homeomorphic to
a subset $\Lambda$ of $\B^{1,\a}\setminus\B_p$ 
such that $\B_p$ is open in $\Lambda\cup\B_p$.
\end{lem}
\begin{proof}
Fix $K\in\B^{1,\a}$ whose support function $h_K$ is not $C^\infty$. 
For $t\in \II$ consider the map $f_t\co\B_p\to \B^{1,\a}$
given by $f_t(D)=tK+(1-t)D$. By~\cite{KraPar} the image of $f_t$
is in $\B^{1,\a}$, but is it not in $\B_p$ for $t\neq 0$
because if $h_{f_t(D)}\in\B_p$ then
$h_K$ is a linear combination of $C^\infty$ functions.
For each $t\neq 0$ the map $f_t$ is a topological embedding.
(Indeed, the map is injective as we can cancel $tK$~\cite[p.48]{Sch-book}, and moreover
if $f_t(D_1)$, $f_t(D_2)$
are close then so are their support functions, and after subtracting $th_K$ we conclude
that the support functions of $(1-t)D_1$, $(1-t)D_2$ are close, and hence so are $D_1$, $D_2$.)
Since $\B_p$ is homeomorphic to
$\Si^\o$ the space $\B^{1,\a}\setminus\B_p$ contains a topological copy of $Q$
which must be closed in $Q\cup\B_p$ since $Q$ is compact.
Any (separable metric) space embeds into $Q$. If $\Lambda$ is the image of
such an embedding into the above copy of $Q$, 
then $\Lambda$ is closed in $\Lambda\cup\B_p$.
\end{proof}

\section{Homeomorphisms of pairs}
\label{sec: pairs}

In this section we finish the proof of Theorem~\ref{thm: intro hyperspace}
by making the homeomorphisms in 
Sections~\ref{sec: hyperspace of dim >1}--\ref{sec: hyp homeo to Si-to-o}
compatible. This is standard but somewhat technical.

If $Y$ is a subspace of $X$, then $(X,Y)$ is a {\em pair\,}. 
A pair $(M, X)$ is {\em $(\M_0, \M_2)$-absorbing\,} if  $(M, X)$ is 
strongly $(\M_0, \M_2)$-universal and $M$ contains a sequence of  
compact subsets $(K_l)_{l\in\o}$ 
such that each $K_l\cap X$ is in $\M_2$ and $\bigcup_{l\in\o} K_l$
is a $\s Z$-set that contains $X$.
(A definition of a
strongly  $(\M_0, \M_2)$-universal pair
can found in~\cite[Section 1.7]{BRZ-book} and
is not essential for what follows). 

\begin{lem} 
\label{lem: pair absorb}
If $U$ is an open subset of a $Q$-manifold $M$, and
$X$ is a homotopy dense subset of $M$ such that $X$ is a $\Si^\o$-manifold 
and $X$ lies in a $\s Z$-subset of $M$, then
$(U, U\cap X)$ is $(\M_0, \M_2)$-absorbing.
\end{lem}
\begin{proof}
Any $\Si^\o$-manifold is strongly $\M_2$-universal and any
$Q$-manifold is a Polish ANR, hence  
by~\cite[Theorem 3.1.3]{BRZ-book} the pair $(U, U\cap X)$
is strongly $(\M_0,\M_2)$-universal.
Since $M$ is $\s$-compact, so is $U$ and any $Z$-set in $M$. 
Thus $U\cap X$ lies in a countable union of compact $Z$-sets in $M$,
and clearly the intersection of $U\cap X$ with any compact subset is in $\M_2$.
\end{proof}

\begin{lem}
\label{lem: absorb pairs ex}
The following pairs $(M,X)$ are $(\M_0, \M_2)$-absorbing:
\begin{enumerate}
\vspace{-2pt}
\item[\textup{(1)}] 
$M=V$ and $V\cap\Si^\o$ where $V$ is any open subset of $Q^\o$,  
\vspace{3pt}
\item[\textup{(2)}] $M=\K_s$ and $X\supset\B_p$ 
such that $X\setminus \B_p$ is $\s$-compact.
\vspace{3pt}
\end{enumerate}
\end{lem}
\begin{proof}
Let us verify the assumptions of Lemma~\ref{lem: pair absorb}.

(1) Since $\Si^\o$ is convex and dense in $Q^{\,\o}$, it is
also homotopy dense in $Q^{\,\o}$, see~\cite[Exercise 13 in 1.2]{BRZ-book}.
Hence $X$ is homotopy dense in $M$.
To show that $X$ lies in a $\s Z$-subset of $M$
let $Q_k\subset Q$ be the set of sequences $(t_i)_{i\in\o}$ with $|t_i|\le 1-\frac{1}{k}$
and \[
N_k=\{
(q_i)\in Q^\o\, : \, q_i=0\ \text{for $i\neq 1$ and $q_1\in Q_k$}.
\}
\]
Since $Q_k$ is a compact subset of the pseudo-interior of $Q$,
it is a $Z$-set in $Q$.
Thus $N_k$ is a $Z$-set in $Q^\o$ and hence $M_k=V\cap N_k$ is a $Z$-set in $V$. 
Finally, $\Si=\bigcup_{k\in\o} Q_k$ implies $\Si^\o\subset\bigcup_{k\in\o} M_k$.

(2) $\B_p$ is homotopy dense in $\K_s$ by Lemma~\ref{lem: regularization}, 
and  contained in a $\s Z$ subset of $\K_s$ by Lemma~\ref{lem: sZ}. 
\end{proof}

The following uniqueness theorem is immediate 
from~\cite[Theorem 1.7.7]{BRZ-book}.

\begin{lem}
\label{lem: uniq absorb}
For $i\in\{1,2\}$ let $(M_i, X_i)$ be a $(\M_0, \M_2)$-absorbing
pair and $B_i\subset M_i$ be a closed subset.
Then for any homeomorphism $f\co M_1\to M_2$ with $f(B_1)=B_2$ 
and $f(B_1\cap X_1)=B_2\cap X_2$ there exists 
a homeomorphism $h\co M_1\to M_2$ such that $h(X_1)=X_2$ and 
$h=f$ on $B_1$.  
\end{lem}

\begin{proof}[Proof of Theorem~\ref{thm: intro hyperspace}]
Recall that $\K_s$ is homeomorphic to $Q^\o\setminus \{*\}$
and $\K_s^{0\le 1}$ is a $Z$-set in $\K_s$ that is disjoint from
$\B_p$. 
Fix an arbitrary $Z$-embedding $\K_s^{0\le 1}\to Q^\o\setminus \{*\}$ 
whose image $B$ is disjoint from $\Si^\o$, e.g., we can fix a factor
in the product $Q^\o$ and pick $B$ in the pseudo-boundary of
that factor. The unknotting of $Z$-sets in $Q$-manifolds~\cite[Theorem 1.1.25]{BRZ-book}
and Lemma~\ref{lem: uniq absorb}
give a homeomorphism $\K_s\to Q^\o\setminus \{*\}$
taking $\K_s^{0\le 1}$ to $B$, and $\B_p$ to $\Si^\o$. 
\end{proof}

\begin{rmk}
\label{rmk: (Q, si-to-o manifolds}
A {\em $(Q^{\,\o}, \Si^\o)$-manifold\,} is a pair $(M, X)$
such that any point of $M$ has a neighborhood $U$ that admits an 
open embedding $h\co U\to Q^{\,\o}$ with $h(U\cap X)=h(U)\cap\Si^\o$.
Lemmas~\ref{lem: pair absorb}--\ref{lem: uniq absorb}
imply that any $(M, X)$ as in Lemma~\ref{lem: pair absorb}
is a $(Q^{\,\o}, \Si^\o)$-manifold.
In fact, a pair $(M, X)$ is a $(Q^{\,\o}, \Si^\o)$-manifold if and only if 
$M$ is a $Q$-manifold and $(M, X)$ is  $(\M_0, \M_2)$-absorbing,
see~\cite[Exercise 12 of Section 1.7]{BRZ-book}.
\end{rmk}

\section{Quotients of hyperspaces}
\label{sec: quotients}

In this section we prove Theorem~\ref{thm: intro quotient} and related results.

\begin{lem}
\label{lem: X/H}
 Let $H\le O(n)$ be a closed subgroup,
$X$ be an $H$-invariant subspace of $\K_s$ that contains $\B_p$,
and $X/H$ be the quotient space. Then
\begin{enumerate}
\vspace{-2pt}
\item[\textup{(1)}]
$X/H$ is a separable metrizable AR.\vspace{2pt}
\item[\textup{(2)}]
$X$ is Polish if and only if so is $X/H$.\vspace{2pt}
\item[\textup{(3)}]
$X$ is locally compact if and only if so is $X/H$. 
\vspace{2pt}
\item[\textup{(4)}] 
If $X=\K_s^{l\le n}$ for some $l\ge 0$, then $X/H$ is 
Polish and locally compact.\vspace{2pt}
\item[\textup{(5)}] 
If $X$ is homeomorphic to $\Si^\o$, then \vspace{2pt}
\begin{enumerate}
\item[\textup{(5a)}] 
any $\s$-compact subset of $X/H$ has empty interior,\vspace{2pt}
\item[\textup{(5b)}] 
$X/H$ is neither Polish nor locally compact.  
\end{enumerate} \end{enumerate} 
\end{lem}
\begin{proof}
(1) 
Convex hulls of finite sets with rational coordinates form a
dense countable subset of $\K_s$. 
Separability and metrizability of a space is inherited by its subsets and
$H$-quotients~\cite[Proposition 1.1.12]{Pal-mem} so that $X/H$
enjoys these properties.

Antonyan~\cite[Theorem 4.5]{Ant-ANE-G-conv} showed that $\K$ has 
what he called an $H$-convex structure. The structure is inherited by
any convex $H$-invariant subset, such as $\K_s$. 
The unit ball $\mathbb B^n$ is a fixed point for
the $H$-action on  $\K_s$, hence~\cite[Theorem 3.3]{Ant-ANE-G-conv}
shows that $\K_s$ is a $H$-AE. 
Hence $\K_s$ is a $H$-AR
because the identity map of any closed $H$-invariant subset extends to an $H$-retraction.
Finally~\cite[Theorem 1.1]{Ant-orbit-sp-ANR} implies that 
$\K_s/H$ is an AR.

Lemma~\ref{lem: regularization} shows
that $X/H$ is a homotopy dense subset $\K_s/H$, an AR,
which makes $X/H$ an AR, see~\cite[Exercise 12 in Section 1.3]{BRZ-book}.

(2) The orbit map $X\to X/H$ is a closed continuous surjection with compact preimages.
For any such map the domain is Polish if and only if so is the co-domain, see 
the references mentioned before Theorem 4.3.27 of~\cite{Eng-gen-top}.
%

(3) The orbit map is proper and open, so the image and the preimage of a compact neighborhood
is a compact neighborhood.

(4) 
$\K_s^{l\le n}$ are homeomorphic to open subsets of $\K_s$
which is a $Q$-manifold, and hence so is $\K_s^{l\le n}$. 
Any $Q$-manifold is Polish and locally compact so (2)-(3) applies.

(5a) If $X/H$ contains a $\s$-compact subset with nonempty interior
then so does $X$ because the orbit map $X\to X/H$ is proper and continuous.
If $V$ is an open set in the interior of a $\s$-compact subset of $\Si^\o$,
then separability of $\Si^\o$ implies that it is covered by countably many translates of $V$
and hence $\Si^\o$ is $\s$-compact. But the product of infinitely many
$\s$-compact noncompact spaces is never $\s$-compact.

(5b) That $X/H$ is not Polish follows from (2) and the remark that $\Si^\o$
is $\s Z$ and hence not Polish by the Baire category theorem because any $Z$-set is nowhere dense.
Also any locally compact (separable) space is $\s$-compact so we are done by (5a).
\end{proof}

Given $X\subset\K_s$ let $\accentset{\circ} X$ denotes the set 
of points of $X$ whose isotropy subgroup
in $O(n)$ is trivial. If $X$ is also $O(n)$-invariant,
$\bm{X}$ denotes the orbit space $X/O(n)$.  
By the slice theorem  
$\accentset{\circ} X$ is open in $X$~\cite[Corollary II.5.5]{Bre-book}
and the orbit map $\accentset{\circ} X\to \accentset{\circ}{\!\bm{X}}$
is a locally trivial principal $O(n)$-bundle~\cite[Corollary II.5.8]{Bre-book}.
The classifying space for such bundles is $BO(n)$, 
the Grassmanian of $n$-planes in $\R^\o$.

\begin{lem}
\label{lem: hom dense trivial isometry}
If $\accentset{\circ}\B_p\subset X\subset\K_s$, then $X$ is homotopy dense in $\K_s$. 
\end{lem}
\begin{proof}
It suffices to show that 
any map $h\co Q\to\K_s$ can be approximated 
by a map with image in $\accentset{\circ}\B_p$. 
Since $\B_p$ is homotopy dense in $\K_s$ we can assume that
$h(Q)\subset\B_p$.

Let $\{e_i\}_{1\le i\le n}$ be the standard basis in $\R^n$.
For $D\in\B_p$ let $e_i(D)$ be the unique point of $\d D$ 
with outward normal vector $e_i$. That $D\in\B_p$
ensures continuity of the map $D\to e_i(D)$.

Pick $\e_1>0$ so small that if $D(\e_1)$ is the 
result of the $(\e_1, e_1)$-truncating of $D\in h(Q)$,
then 
\[D\setminus D(\e_1)\qquad \text{and}\qquad \bigcup_{1<l\le n}\{e_l(D)\}\]
have disjoint closures.
The $(\e_1, e_1)$-truncating turns $D$ into a convex body $D(\e_1)$
whose flat face $F_1(D)$ has normal vector $e_1$.
Continuing inductively pick $\e_k>0$ so that if $D(\e_k)$ is the 
result of the $(\e_k, e_k)$-truncating of $D(\e_{k-1})$
then the three sets 
\[
D(\e_{k-1})\setminus D(\e_k)\qquad \quad
\bigcup_{1\le j<k} F_j(D)\quad\qquad \bigcup_{k<l\le n}\{e_l(D)\}\]
have disjoint closures,
and the diameter of the newly formed flat face $F_k(D)$ of $D(\e_k)$ is smaller
than the minimum over $D\in h(Q)$ of the diameters of $F_{k-1}(D)$.

Compactness of $h(Q)$ implies that any small enough $\e_k$
works for all $D\in h(Q)$ at once.  Set 
$\e=(\e_1,\dots, \e_n)$. 
Let $h_\e\co Q\to\K_s$ be the composition of $h$ with the map $D\to D(\e_n)$, where $D\in h(Q)$.
Note that $h_\e$ converges to $h$ as $\e\to o$. Any $\phi\in O(n)$ that preserves $D(\e_n)$
must take faces to faces, and since the faces all
have different diameters $\phi$ must preserve each face
and hence its normal vector, so that $\phi$ is the identity.
Thus the image of $h_\e$ lies in $\accentset{\circ}\K_s$ which is an open subset of $\K_s$.
Finally, using homotopy density of $\B_p$ we can approximate $h_\e$ by a map with image
in $\B_p\cap\accentset{\circ}\K_s=\accentset{\circ}\B_p$.
\end{proof}

\begin{lem}
\label{lem: BO(n)}
If $\accentset{\circ}\B_p\subset X\subset\K_s$ and $\accentset{\circ} X$ is
$O(n)$-invariant, then
$\accentset{\circ} X$ is contractible and $\accentset{\circ}{\!\bm{X}}$
is homotopy equivalent to $BO(n)$.
\end{lem}
\begin{proof}
 Lemma~\ref{lem: hom dense trivial isometry} gives a homotopy equivalence of 
$\accentset{\circ} X$ and $\K_s$, so that
$\accentset{\circ} X$ is contractible.
Hence $\accentset{\circ}{\!\bm{X}}$ is the classifying space for the principal
$O(n)$-bundles, and thus it is homotopy equivalent to $BO(n)$, see~\cite[Section 7]{Dol-partition}. 
\end{proof}

\begin{rmk}
In Lemma~\ref{lem: BO(n)} 
if $X$ is homeomorphic to $\Si^\o$, then so is $\accentset{\circ} X$
because $\accentset{\circ} X$ is open in $X$ and
homotopy equivalent $\Si^\o$-manifolds are homeomorphic. Similarly, if $X$ is $Q$-manifold,
then so is $\accentset{\circ} X$, even though they might be non-homeomorphic. 
\end{rmk}

\begin{thm} 
\label{thm: quotient sigma-to-omega}
Suppose $\accentset{\circ}\B_p\subset X\subset\B^{1,\a}$ with $\a>0$ and 
$\accentset{\circ} X$ is $O(n)$-invariant. 
If $\accentset{\circ} X$ is a $\Si^\o$-manifold, then so is $\accentset{\circ}{\!\bm{X}}$. 
\end{thm}
\begin{proof}
By local triviality of the bundle
$\accentset{\circ} X\to \accentset{\circ}{\bm{X}}$ each point of 
$\accentset{\circ}{\!\bm{X}}$ has a neighborhood that $U$ such that $U\times O(n)$
is a $\Si^\o$-manifold. Hence $U\times \R^{d}$ is a $\Si^\o$-manifold where $d=\dim O(n)$.

Let us show that $U$ is a $\Si^\o$-manifold.
Note that $U$ is an ANR as a retract of the ANR $U\times \R^{d}$. Also 
$U$ satisfies SDAP by~\cite[Exercise 9 in section 1.3]{BRZ-book}.
Strong $\M_2$-universality of $U$ comes from~\cite[Theorem 3.2.18]{BRZ-book}, namely,
since $U\times \R^{d}$ is strongly $\M_2$-universal,
so is $U\times \R^{d}\times\R^\o$, and hence $U$. 
Since $\M_2$ is closed-hereditary and $U\times \R^{d}$ is in $\M_2$, so is $U$.

Remark~\ref{rmk: invariant Z-set} shows that $X$
is a countable union of $O(n)$-invariant $Z$-sets, and hence
so is $\accentset{\circ} X$
because $\accentset{\circ} X$ is $O(n)$-invariant and open, while
the intersection of a $Z$-set with any open is a $Z$-set in that 
open set~\cite[Corollary 1.4.5]{BRZ-book}. 

The bundle projection
$\pi\co\accentset{\circ} X\to \accentset{\circ}{\!\bm{X}}$ maps
any $O(n)$-invariant $Z$-set $\check Z$ to a $Z$-set $\pi(\check Z)$. 
Indeed, since $Q$ is contractible, 
any map $Q\to\accentset{\circ}{\!\bm{X}}$ lifts to a map $Q\to \accentset{\circ} X$
which since $\accentset{\circ} X$ is open
can be approximated by a map $h\co Q\to \accentset{\circ} X$ that misses $\check Z$. 
Now $\pi\circ h$ misses $\pi(\check Z)$ due to $O(n)$-invariance of $\check Z$.

Therefore $\accentset{\circ}{\!\bm{X}}$ is $\s Z$, and hence
so is $U$, again by~\cite[Corollary 1.4.5]{BRZ-book}. 
Thus $U$ is $\M_2$-absorbing and hence
is a $\Si^\o$-manifold. Thus $\accentset{\circ}{\bm{X}}$ is a $\Si^\o$-manifold.
\end{proof}

\begin{thm}
\label{thm: quotient homeo to Q} 
Suppose $\accentset{\circ}\B_p\subset X\subset\K_s$ and 
$\accentset{\circ} X$ is $O(n)$-invariant. 
If $\accentset{\circ} X$ is a $Q$-manifold, then  so is $\accentset{\circ}{\!\bm{X}}$. 
\end{thm}
\begin{proof}
As in the proof of Lemma~\ref{lem: X/H} we see that $\accentset{\circ} X$
is locally compact. As in the proof of Theorem~\ref{thm: quotient sigma-to-omega} 
we conclude that any point of $\accentset{\circ}{\!\bm{X}}$ has an ANR neighborhood,
hence $\accentset{\circ}{\!\bm{X}}$ is an ANR. By Toru{\'n}czyk characterization of
$Q$-manifolds among locally compact ANRs~\cite[Theorem 1]{Tor-char-Q} 
it remains to check that any map $f\co \II^k\to \accentset{\circ}{\!\bm{X}}$
can be approximated by a map whose image is a $Z$-set. By Lemma~\ref{lem: regularization}
$\bm{\B}_p$ is homotopy dense in $\bm{\K}_s$, so we  
approximate $f$ by a map with image in $\bm{\B}_p$, which actually lies
in $\accentset{\circ}{\bm{\B}}_p=\bm{\B}_p\cap \accentset{\circ}{\!\bm{X}}$ because
$\accentset{\circ}{\!\bm{X}}$ is open. Since 
$\accentset{\circ}{\!\bm{X}}$ is a $\Si^\o$-manifold, any 
compactum in $\accentset{\circ}{\!\bm{X}}$ is a $Z$-set~\cite[Proposition 1.4.9]{BRZ-book}.
\end{proof}

\begin{thm}
\label{thm: abs pairs local}
Suppose $\accentset{\circ}\B_p\subset X\subset\B^{1,\a}$ with $\a>0$ and 
$\accentset{\circ} X$ is $O(n)$-invariant. 
If $\accentset{\circ} X$ is a $\Si^\o$-manifold, then each point of 
$\accentset{\circ}{\!\bm{X}}$ has a neighborhood $U$ in $\accentset{\circ}{\bm{\K}}_s$
such that there is an open embedding $h\co U\to Q^{\,\o}$ with $h(U\cap X)=h(U)\cap\Si^\o$. 
\end{thm}
\begin{proof}
As in the proof of Theorem~\ref{thm: quotient sigma-to-omega}
we use Remark~\ref{rmk: invariant Z-set}
to show that $\accentset{\circ}{\!\bm{X}}$ lies in the $\s Z$-subset of 
$\accentset{\circ}{\bm{\K}}_s$. 
Now Remark~\ref{rmk: (Q, si-to-o manifolds}
and Theorems~\ref{thm: quotient sigma-to-omega}--\ref{thm: quotient homeo to Q}
imply that $(\accentset{\circ}{\bm{\K}}_s, \accentset{\circ}{\!\bm{X}})$ is a
$(Q^{\,\o}, \Si^\o)$-manifold.
\end{proof}

\begin{thm}
Let $L$ be the product of $[0,1)$ and any locally finite simplicial complex
homotopy equivalent to $BO(n)$.
If $\accentset{\circ}\B_p\subset X\subset\B^{1,\a}$ with $\a>0$ such that 
$\accentset{\circ} X$ is an $O(n)$-invariant $\Si^\o$-manifold, then there is a homeomorphisms
$h\co \accentset{\circ}{\bm{\K}}_s\to L\times Q^{\,\o}$ that takes 
$\accentset{\circ}{\!\bm{X}}$ onto $L\times \Si^\o$.
\end{thm}
\begin{proof}
Theorem~\ref{thm: quotient homeo to Q} says that $\accentset{\circ}{\bm{\K}}_s$ is a $Q$-manifold,
which  by Lemma~\ref{lem: prod with [0,1)} is homeomorphic to its product with $[0,1)$.
The product of any locally finite simplicial complex with $Q$ is a 
$Q$-manifold~\cite[Theorem 1.1.24]{BRZ-book}.
Lemma~\ref{lem: BO(n)} gives a homotopy equivalence 
of $\accentset{\circ}{\bm{\K}}_s$ and $L\times Q$, and
since both spaces are products of $[0,1)$ and a $Q$-manifold,
they are homeomorphic~\cite[Theorem 23.1]{Cha-book}.
Also $\accentset{\circ}{\!\bm{X}}$, $L\times \Si^\o$
are $\Si^\o$-manifolds, see Theorem~\ref{thm: quotient sigma-to-omega} 
and Section~\ref{sec: inf dim top}. The pair
$(\accentset{\circ}{\bm{\K}}_s, \accentset{\circ}{\!\bm{X}})$  is a 
$(Q^{\,\o},\Si^\o)$-manifolds by Theorem~\ref{thm: abs pairs local}, 
and hence is $(\M_0, \M_2)$-absorbing, see Remark~\ref{rmk: (Q, si-to-o manifolds}. 
As a $\s$-compact space $L$ lies in $\M_2$, and then
the proof of Lemma~\ref{lem: absorb pairs ex}(1) shows that
$(L\times Q^{\,\o}, L\times \Si^\o)$ is $(\M_0, \M_2)$-absorbing.
The claim now follows from Lemma~\ref{lem: uniq absorb}. 
\end{proof}

\begin{rmk}
To make $L$ explicit start with the standard CW structure on $BO(n)$
given by Schubert cells, use mapping telescope to replace it by a 
homotopy equivalent locally finite CW complex, and triangulate the result.
\end{rmk}

\section{Deforming disks modulo congruence}
\label{sec: destroy}

The results of Section~\ref{sec: quotients} say little about
the local structure of  $\bm{\K}_s$ near the points of
$\bm{\K}_s\setminus\accentset{\circ}{\bm{\K}}_s$. For example, 
one wants to have
better understanding of the standard stratification of $\bm{\K}_s$
by orbit type. As was mentioned in the introduction the stratum
$\accentset{\circ}{\bm{\K}}_s$ is not homotopy dense in
$\bm{\K}_s$ because they have different fundamental groups, so
a singular $2$-disk in $\bm{\K}_s$ cannot always be pushed  
into $\accentset{\circ}{\bm{\K}}_s$. This is possible 
for any path in $\bm{\K}_s$
because it lifts into $\K_s$~\cite[Theorem II.6.2]{Bre-book} 
and then Lemma~\ref{lem: hom dense trivial isometry}
applies. The lemma below shows that under mild assumptions
on a point $\bm{x}\in \bm{\K}_s\setminus\accentset{\circ}{\bm{\K}}_s$
there is a small singular disk in $\bm{\K}_s$ near $\bm{x}$ that cannot be pushed into 
$\accentset{\circ}{\bm{\K}}_s$. Thus the orbit space analog of
Lemma~\ref{lem: hom dense trivial isometry} fails locally.

Let $\pi_j$ denote the $j$th homotopy group, 
let $\pi\co \K_s\to\bm{\K}_s$ be orbit map,  
$I_x$ be the isotropy subgroup of $x$ in $O(n)$, and $\bm{x}=\pi(x)$.

\begin{lem}
\label{lem: local no sym destroy}
Let $\B_p\subset X\subset \K_s$ where $X$ is $O(n)$-invariant. 
If \mbox{$x\in X\setminus \accentset{\circ}{X}$} and
the inclusion $I_x\to O(n)$ is $\pi_k$-nonzero for some $k\ge 0$,
then any neighborhood of $\bm{x}\in\bm{X}$ contains a neighborhood $\bm{U}$
such that the inclusion $\accentset{\circ}{\bm{X}}\cap\bm{U}\to\bm{U}$
is not $\pi_{k+1}$-injective, and in particular,
$\pi_k(\bm{U}, \accentset{\circ}{\bm{X}}\cap\bm{U})\neq 0$.
\end{lem}
\begin{proof}
Inside any neighborhood of $x$ in $X$ one can
find a neighborhood of $x$ of the form $U=X\cap V$ where $V$ is 
$I_x$-invariant, convex and open in $\K_s$ (e.g., let $V$
be a sufficiently small ball about $x$ in the Hausdorff metric).
The orbit $I_x(y)$
of any $y\in U\cap \accentset{\circ}{X}$
lies $U\cap \accentset{\circ}{X}$. 
Let $f\co S^k\to I_x(y)$
be a map that is $\pi_k$-nonzero when composed with the inclusion
$I_x(y)\to O(n)$. 
Since $V$ is convex it contains the family 
$f_t\co \mathbb S^k\to V$ of singular $k$-spheres given by
$f_t(v)=tf(v)+(1-t)x$, $t\in\II$.

Since $\accentset{\circ}{X}$ is homotopy dense in $\K_s$ we can slightly deform
$f_t$ to $h_t\co \mathbb S^k\to \accentset{\circ}{X}\cap V$ with $h_1=f$. 
Define
$h\co \mathbb B^{k+1}\to\accentset{\circ}{X}\cap V$ by $h(tv)=h_t(v)$.
Since $f$ is $\pi_k$-nonzero in the fiber
of the bundle $\pi\co \accentset{\circ}{\K}_s\to \accentset{\circ}{\bm{\K}}_s$
the singular $(k+1)$-sphere map $\pi\circ h$  in
$\accentset{\circ}{\bm{X}}\cap\bm{U}$ is not null-homotopic 
in $\accentset{\circ}{\bm{\K}}_s$.

Now $r\to rh(tv)+(1-r)x$, $r\in\II$, is 
a family of singular $(k+1)$-disks in $V$ whose boundaries are
the singular $k$-spheres $r\to rf(v)+(1-r)x$ each projecting
to a point of $\bm{\K}_s$. Applying $\pi$
gives a null-homotopy of the singular $(k+1)$-sphere
$\pi\circ h$ inside $\bm{V}$, and hence
in $\bm{U}=\bm{X}\cap\bm{V}$ by Schneider's regularization.
\end{proof}

\begin{rmk}
The closed subgroups $H$ for which the inclusion 
$H\to O(n)$ is zero on all homotopy groups can be completely 
understood, see~\cite{Bry-MO}.
\end{rmk}

{\bf Acknowledgments:} I am grateful to Sergey Antonyan,
Taras Banakh, Robert Bryant, Mohammad Ghomi, Mikhail Ostrovskii,
and Rolf Schneider for helpful discussions.

\small
\bibliographystyle{amsalpha}
\bibliography{cb}

\end{document}